\documentclass[11pt]{article}
\usepackage[english]{babel}
\usepackage{epsfig}
\usepackage{amsmath}
\usepackage{amsthm}
\usepackage{amssymb}
 
\newtheorem{theorem}{Theorem}[section]
\newtheorem{definition}[theorem]{Definition}

\newtheorem{proposition}[theorem]{Proposition}
\newtheorem{corollary}[theorem]{Corollary}
\newtheorem{remark}[theorem]{Remark}

\begin{document}

\title{\bf Octonionic Kerzman-Stein operators}

\author{
Denis Constales \\
Department of Electronics and Information Systems\\
Faculty of Engineering and Architecture\\
Ghent University\\
Krijgslaan 281-S8\\
9000 Gent, Belgium\\
Denis.Constales@UGent.be\\
\and
Rolf S\"{o}ren Krau{\ss}har\\
Fachbereich Fachwissenschaft Mathematik\\
Erziehungswissenschaftliche Fakult\"at\\
Universit\"at Erfurt\\
Nordh\"auser Str. 63\\
99089 Erfurt, Germany\\
soeren.krausshar@uni-erfurt.de }

\maketitle
\begin{abstract}  
In this paper we consider generalized Hardy spaces in the octonionic setting associated to arbitrary Lipschitz domains where the unit normal field exists almost everywhere. First we discuss some basic properties and explain structural differences to the associative Clifford analysis setting. 
The non-associativity requires special attention in the definition of an appropriate inner product and hence in the definition of a  generalized Szeg\"o projection. Whenever we want to apply classical theorems from reproducing kernel Hilbert spaces we first need to switch to the consideration of real-valued inner products where the Riesz representation theorem holds.  
Then we introduce a generalization of the dual Cauchy transform for octonionic monogenic functions which represents the adjoint transform with respect to the real-valued inner product $\langle  \cdot, \cdot \rangle_0$ together with an associated octonionic Kerzman-Stein operator and related kernel functions.  

Also in the octonionic setting, the Kerzman-Stein operator that we introduce turns out to be a compact operator. A motivation behind this approach is to find an approximative method to compute the Szeg\"o projection of octonionic monogenic functions offering a possibility to tackle BVP in the octonions without the explicit knowledge of the octonionic Szeg\"o kernel which is extremely difficult to determine in general. We also discuss the particular cases of the octonionic unit ball and the half-space. Finally, we relate our octonionic Kerzman-Stein operator to the Hilbert transform and particularly to the Hilbert-Riesz transform in the half-space case. 

\end{abstract}
\noindent {\bf Mathematical Review Classification numbers}: 30G35\\

\section{Introduction}  
In the recent years one noted an increasing interest in the study of generalizations of Bergman and Hardy spaces in the setting of octonionic monogenic function theory. 
\par\medskip\par
In \cite{WL2018} Jinxun Wang and Xingmin Li determined the Bergman and the Szeg\"o kernel for octonionic monogenic functions on the unit ball. In their follow-up paper  \cite{WL2020} they proved a representation formula for the octonionic Bergman kernel of upper half-space. In our recent paper \cite{Kra2020-2} we managed to set up an explicit formula for the octonionic Szeg\"o kernel of upper half-space as well as for the octonionic Szeg\"o and Bergman kernel of strip domains that are bounded in the real direction. Our method used octonionic generalizations of the cotangent and cosecant series. However, due to the lack of associativity in the octonions, our proof relies on the property that the strip domains that we considered are bounded in the real direction only. In one part of our proof we explicitly exploited that a product of three elements $a,b,c \in \mathbb{O}$ where one of these factors lies in the real axis is associative. This is not the case anymore when other directions than the real one are involved. To determine explicit formulas for the Bergman and the Szeg\"o kernel therefore is even more difficult than in the associative case of working in Clifford algebras. 
\par\medskip\par
Note that in classical function theory the Szeg\"o projection which involves the Szeg\"o kernel plays a crucial role in the resolution of singular boundary value problems for octonionic monogenic functions. 
 \par\medskip\par
However, already in the 1970s one discovered in the context of complex and harmonic analysis that there is an alternative possibility to evaluate the Szeg\"o projection without having an explicit formula for the Szeg\"o kernel, namely one can use Kerzman-Stein operators, see for instance the classical references \cite{Bell,KS}. Afterwards, Kerzman-Stein theory has been generalized extensively to the associative Clifford analysis setting, see for example \cite{BL,BD,Ca,Cn,Delanghe,PVL}.
\par\medskip\par
In \cite{XLT2008} Xingmin Li, Zhao Kai and Qian Tao successfully introduced a Cauchy transform in the octonionic setting and were able to set up related Plemelj projection formulas together with a basic toolkit to study operators of Calderon-Zygmund type acting on octonionic monogenic functions defined on some Lipschitz surfaces. See also the more recent paper \cite{Kheyfits} where further connections to harmonic analysis are addressed. 
\par\medskip\par
In this paper we present an attempt to introduce a generalization of the dual Cauchy transform  in the octonionic setting together with an associated Kerzman-Stein operator and a related octonionic Kerzman-Stein kernel. The lack of associativity needs to be carefully taken into account and requires particular attention and arguments. 

In particular, the non-associativity requires a re-definition as well as a different interpretation of the classical constructions. A crucial need is to properly adapt the definition of an appropriately inner product on the corresponding Hardy space of octonionic monogenic functions. 

First we note that in \cite{WL2018,WL2020} the authors used two different definitions of an octonion-valued inner product; one particular definition for the unit ball setting and another one for the half-space setting. An open question was to find out a general explanation for that necessity and to figure out a general scheme behind these particular choices.

After having introduced the basic notions in Section 2, in Section 3 we turn to the question how to  carefully define octonionic monogenic Hardy spaces for general Lipschitz domains that have a smooth boundary almost everywhere. We introduce  two definitions of an inner product that can be applied for all these domains. In the particular setting of the unit ball, the octonionic inner product that we first consider coincides exactly with the particular one considered in \cite{WL2018}. In the half-space setting it also coincides with the particular definition given in \cite{WL2020}. So, we understand how these different definitions arise and how they fit together within a general theory. Furthermore, in the case of associativity that particular inner product always coincides with the inner product considered in complex and Clifford analysis. Additionally, we prove that the octonionic Hardy space really has always a continuous point evaluation.  However, special attention is required when addressing properties like orthogonality or a Fischer-Riesz representation theorem which in turn is essential in the definition of an adjoint. Here, we cannot directly work with the octonionic inner product.  Due to the lack of the $\mathbb{O}$-linearity which in turn is a consequence of the non-associativity, we do not have a Cauchy-Schwarz inequality in general. One possible way to overcome this problem is to work with a real-valued inner product that may be derived from the octonionic one by taking its scalar part whenever we need to apply classical theorems for reproducing kernel spaces. In terms of this inner product all the classical notions then  can be well-defined. 


As a consequence of the non-associativity, the construction of an adjoint Cauchy transform is a non-trivial problem, too. Furthermore, it crucially relies on the particular definition of the special inner product. The adjoint that we present in this paper has to be understood in the sense of the real-valued inner product $\langle \cdot, \cdot \rangle_0$. 
Nevertheless, all our constructions are completely compatible with the classical ones as soon as one has  associativity. 

\par\medskip\par

After having introduced a Kerzman-Stein kernel we prove some basic properties of the related octonionic Kerzman-Stein operators. It is a skew symmetric operator and the kernel vanishes exactly if and only if the domain is the octonionic unit ball. In fact our proposal for a dual Cauchy transform coincides with the Cauchy transform exactly and exclusively in the case of the unit ball providing us with a nice analogy to the classical theory.  

\par\medskip\par

Furthermore, we show that this octonionic version of the Kerzman-Stein operator is a compact operator. 

Compactness is a key ingredient if we want to develop an approximative construction method to evaluate the octonionic Szeg\"o projection and to compute the Szeg\"o kernel functions purely relying on the global Cauchy kernel and the particular geometry of the boundary of the domain.
 
\par\medskip\par

Again we pay special attention to the particular context of the octonionic unit ball and octonionic upper half-space. Finally, we relate the octonionic Kerzman-Stein operator to the Hilbert transform and particularly to the Hilbert-Riesz transform in the half-space case.
   
\section{Preliminaries}

\subsection{Basics on octonions}

The famous theorem of Hurwitz tells us that $\mathbb{R}$, $\mathbb{C}$, the Hamiltonian skew field of the quaternions $\mathbb{H}$ and the octonions $\mathbb{O}$ invented by Graves and Cayley are the only real normed division algebras up to isomorphy. The octonions represent an eight-dimensional real non-associative algebra over the field of real numbers.  
Following for instance \cite{Baez,Ward} and many other classical references, one can construct the octonions by applying the so-called Cayley-Dickson doubling process. To leave it simple, let us take two pairs of complex numbers $(a,b)$ and $(c,d)$. Then one defines an addition and multiplication operation on these pairs by  
$$
(a,b)+(c,d) :=(a+c,b+d),\quad\quad (a,b)\cdot (c,d) := (ac-d\overline{b},\overline{a}d+cb) 
$$ 
where $\overline{\cdot}$ represents the classical complex conjugation. Subsequentially, this automorphism is extended to an anti-automorphism by defining $\overline{(a,b)}:=(\overline{a},-b)$ on this set of pairs of numbers $(a,b)$. We just have constructed the real Hamiltonian quaternions $\mathbb{H}$. Each quaternion can be written as $x=x_0 + x_ 1e_1 + x_2 e_2 + x_3 e_3$ where $e_i^2=-1$ for $i=1,2,3$. Furthermore, we have $e_1 e_2 = e_3$, $e_2 e_3 = e_1$, $e_3 e_1 = e_2$ and $e_i e_j = - e_j e_i$ for all mutually  distinct $i,j \in \{1,2,3\}$ like for the usual vector product on $\mathbb{R}^3$-vectors. Already this relation exhibits that $\mathbb{H}$ is not commutative anymore, but it is still associative.

After applying once more this duplication process (now on pairs of quaternions), then one has constructed the octonions $\mathbb{O}$.  In real coordinates these can be expressed in the form  
 $$
 x = x_0 + x_1 e_1 + x_2 e_2 + x_3 e_3 + x_4 e_4 + x_5 e_5 + x_6 e_6 + x_7 e_7
 $$
 where $e_4=e_1 e_2$, $e_5=e_1 e_3$, $e_6= e_2 e_3$ and $e_7 = e_4 e_3 = (e_1 e_2) e_3$. 
 Like for quaternions, we also have $e_i^2=-1$ for all $i =1,\ldots,7$ and $e_i e_j = -e_j e_i$ for all mutual distinct $i,j \in \{1,\ldots,7\}$. The way how the octonionic multiplication works is easily visible from the following  table    
\begin{center}
 \begin{tabular}{|l|rrrrrrr|}
 $\cdot$ & $e_1$&  $e_2$ & $e_3$ & $e_4$ & $e_5$ & $e_6$  & $e_7$ \\ \hline
 $e_1$  &  $-1$ &  $e_4$ & $e_5$ & $-e_2$ &$-e_3$ & $-e_7$ & $e_6$ \\
 $e_2$ &  $-e_4$&   $-1$ & $e_6$ & $e_1$ & $e_7$ & $-e_3$ & $-e_5$ \\
 $e_3$ &  $-e_5$& $-e_6$ & $-1$  & $-e_7$&$e_1$  & $e_2$  & $e_4$ \\
 $e_4$ &  $e_2$ & $-e_1$ & $e_7$ & $-1$  &$-e_6$ & $e_5$  & $-e_3$\\
 $e_5$ &  $e_3$ & $-e_7$ & $-e_1$&  $e_6$&  $-1$ & $-e_4$ & $e_2$ \\
 $e_6$ &  $e_7$ &  $e_3$ & $-e_2$& $-e_5$& $e_4$ & $-1$   & $-e_1$ \\
 $e_7$ & $-e_6$ &  $e_5$ & $-e_4$& $e_3$ & $-e_2$& $e_1$  & $-1$ \\ \hline 	
 \end{tabular}
\end{center}
In contrast to a lot of other papers we label the multiplicative algebraic independent units by $e_1$, $e_2$ and $e_3$ and use the same notation as in \cite{Baez}. Note that in this notation the quaternions are represented by the elements of the form $x = x_0 + x_1 e_1 + x_2 e_2 + x_4 e_4$, since in our definition we have $e_4:=e_1 e_2$. Here $e_3$ in another algebraic independent unit. To present expressions like sums in a compact form we also formally use the notation $e_0:=1$. 

As one can also verify by means of this table, we have lost the associativity. Nevertheless, we still deal with a division algebra. Furthermore, the octonions satisfy the alternative property and they still form a composition algebra. 

We have the Moufang rule $(ab)(ca) = a((bc)a)$ holding for all $a,b,c \in \mathbb{O}$. Taking especially $c=1$, then one obtains the flexibility condition $(ab)a= a(ba)$.  

Let $a = a_0 + \sum\limits_{i=1}^7 a_i e_i$ be an element of $\mathbb{O}$. We call $\Re(a):=a_0$ the real part of $a$. $a_i$ will be called the $i$-part of $a$ in the sequel.

The inherited conjugation map imposes the properties $\overline{e_j}=-e_j$ for all $j =1,\ldots,7$ while it leaves the real component invariant, i.e. we have $\overline{a_0}=a_0$ for all $a_0 \in \mathbb{R}$. Applying the conjugation to the product of two octonions $a,b \in \mathbb{O}$ then one gets $\overline{a b} = \overline{b}\; \overline{a}$, like in the quaternionic setting.  

The Euclidean norm and standard scalar product from $\mathbb{R}^8$ can be expressed in the octonionic setting in the way $$
\langle a,b \rangle := \sum\limits_{i=0}^7 a_i b_i = \Re\{a \overline{b}\} = \Re\{\overline{a} b\}
$$ 
and $|a|:= \sqrt{\langle a,a\rangle} = \sqrt{\sum\limits_{i=0}^7 a_i^2}$. The norm composition property $|a \; b|=|a||b|$ holds for all $a,b \in \mathbb{O}$. Every non-zero octonion $a \in \mathbb{O}$ is invertible with $a^{-1} =\overline{a}/|a|^2$, which means that there are no zero-divisors in $\mathbb{O}$.   

Another important octonionic calculation rule is the identity 
\begin{equation}\label{dieckmann}
(a\overline{b})b = \overline{b}(ba) =a(\overline{b}b)=a(b \overline{b})
\end{equation}  
which is true for all $a,b \in \mathbb{O}$ and, 
$\Re\{b(\overline{a}a)c\} =\Re\{(b \overline{a})(ac)\}$ for all $a,b,c \in \mathbb{O}$. An explicit proof is presented for example in \cite{CDieckmann} Proposition 1.6. Analogously, one can prove that $(a b)b = b(ba) =a(bb)$. 
Another fundamental property that we are going to use is that all $a,b,c \in \mathbb{O}$ satisfy 
\begin{equation}
\label{scalarprod}
\langle ab,c \rangle = \langle b, \overline{a}c \rangle,
\end{equation}
cf. \cite{XLT2008} Corollary~3.5. This property will be of crucial importance for the existence of an octonionic adjoint operator in the context of real-valued inner products.  
\par\medskip\par
We also use the notation $B_8(p,r) :=\{x \in \mathbb{O} \mid |x-p| < r\}$ (resp. $\overline{B_8(p,r)} :=\{x \in \mathbb{O} \mid |x-p| \le r\}$) for the eight-dimensional solid open (resp.  closed) ball of radius $r$ centered around $p$ in the octonions. By $S_7(p,r)$ we address the seven-dimensional sphere $S_7(p,r) :=\{x \in \mathbb{O} \mid |x-p| = r\}$. If $x=0$ and $r=1$ then we simply denote the unit ball and the unit sphere by $B_8$ and $S_7$, respectively. The notation $\partial B_8(p,r)$ means the same as $S_7(p,r)$ throughout the whole paper.   
 
\subsection{Basics on octonionic monogenic function theory}
In this subsection we summarize the most important function theoretic properties. Like in the context of quaternions and Clifford algebras, also the octonions offer different approaches to introduce generalizations of complex function theory. 

From \cite{DS,Nono,XL2000} and elsewhere we recall 
\begin{definition}
	Let $U \subseteq \mathbb{O}$ be an open set. Then a real differentiable function $f:U \to \mathbb{O}$ is called left (resp. right) octonionic monogenic if ${\cal{D}} f = 0$ (resp. $f {\cal{D}} = 0$). Here, $
	{\cal{D}}:= \frac{\partial }{\partial x_0} + \sum\limits_{i=1}^7 e_i \frac{\partial }{\partial x_i}$ denotes the octonionic Cauchy-Riemann operator, where $e_i$ are the octonionic units introduced above. If $f$ satisfies $\overline{{\cal{D}}}f = 0$ (resp. $f\overline{\cal{D}} = 0$) we call $f$ left (resp. right) octonionic anti-monogenic. 
\end{definition}
In contrast to quaternionic and Clifford analysis, the set of left (right) octonionic monogenic functions does neither form a right nor a left ${\mathbb{O}}$-module. Following \cite{KO2019}, a simple counterexample can be presented by taking the function $f(x):= x_1 - x_2 e_4$. It satisfies ${\cal{D}}[f(x)] = e_1 - e_2 e_4 = e_1 - e_1 = 0$. However, $g(x):=(f(x))\cdot e_3 = (x_1 - x_2 e_4) e_3 = x_1 e_3 - x_2 e_7$ satisfies ${\cal{D}}[g(x)] = e_1 e_3 - e_2 e_7 = e_5 -(-e_5) = 2 e_5 \neq 0$. The lack of associativity obviously destroys the modular structure of octonionic monogenic functions which already represents one substantial difference to Clifford analysis. Clifford analysis in $\mathbb{R}^8$ and octonionic analysis are essentially different function theories, see also \cite{KO2018}. 

However, alike in Clifford analysis, also octonionic monogenic functions satisfy a Cauchy integral theorem, cf.  for instance  \cite{XL2002}. 
 \begin{proposition}\label{cauchy} (Cauchy's integral theorem)\\
Let $G \subseteq \mathbb{O}$ be a bounded $8$-dimensional connected star-like domain with an orientable strongly Lipschitz boundary $\partial G$. Let $f \in C^1(\overline{G},\mathbb{O})$. If $f$ is left (resp. right) octonionic monogenic inside of $G$, then 
$$
\int\limits_{\partial G} d\sigma(x) f(x) = 0,\quad {\rm resp.}\;\;\int\limits_{\partial G} f(x) d\sigma(x) = 0
$$  	
where $d\sigma(x) = \sum\limits_{i=0}^7 (-1)^j e_i \stackrel{\wedge}{d x_i} = n(x) dS(x)$, where $\stackrel{\wedge}{dx_i} = dx_0 \wedge dx_1 \wedge \cdots dx_{i-1} \wedge dx_{i+1} \cdots \wedge dx_7$ and where $n(x)$ is the outward directed unit normal field at $x \in \partial G$ and $dS(x) =|d \sigma(x)|$ the ordinary scalar surface Lebesgue measure of the $7$-dimensional boundary surface.  
 \end{proposition} 
Following \cite{XL2000}, another structural difference to Clifford analysis is reflected in the lack of a direct analogue of the Stokes formula. Even in the cases where simultaneously ${\cal{D}} f = 0$ and $g {\cal{D}} = 0$ holds, we do not have in general that 
 	$$
 	\int\limits_{\partial G} g(x) \; (d\sigma(x) f(x)) = 0 \quad {\rm nor}\quad 
	\int\limits_{\partial G} (g(x) d\sigma(x)) \; f(x) = 0.
 	$$
Again, the obstruction to get such an identity in general is caused by the non-associativity. Following \cite{XLT2008} one has 
$$
\int\limits_{\partial G} g(x) \; (d\sigma(x)  f(x)) = \int\limits_G \Bigg(   
g(x)({\cal{D}} f(x)) + (g(x){\cal{D}})f(x) - \sum\limits_{j=0}^7 [e_j, {\cal{D}}g_j(x),f(x)]
\Bigg) dV
$$
where $[a,b,c] := (ab)c - a(bc)$ stands for the associator of three octonionic elements. 
A very particular situation is obtained when inserting for $g$ the left and right octonionic monogenic Cauchy kernel $$q_{\bf 0}: \mathbb{O} \backslash\{0\} \to \mathbb{O},\;q_{\bf 0}(x) := \frac{x_0 - x_1 e_1 - \cdots - x_7 e_7}{(x_0^2+x_1^2+\cdots + x_7^2)^4} = \frac{\overline{x}}{|x|^8}.$$ 

 From \cite{Nono,XL2002} and elsewhere we may recall:
 \begin{proposition}\label{cauchy1}(Cauchy's integral formula).\\
Let $U \subseteq \mathbb{O}$ be a non-empty open set and $G \subseteq U$ be an $8$-dimensional compact oriented manifold with a strongly Lipschitz boundary $\partial G$. If $f: U \to \mathbb{O}$ is left (resp. right) octonionic monogenic, then for all $x \not\in \partial G$
$$
\chi(x)f(x)= \frac{3}{\pi^4} \int\limits_{\partial G} q_{\bf 0}(y-x) \Big(d\sigma(y) f(y)\Big),\quad\quad \chi(x) f(x)= \frac{3}{\pi^4} \int\limits_{\partial G}   \Big(f(y)d\sigma(y)\Big) q_{\bf 0}(y-x),
$$
where $\chi(x) = 1$ if $x$ is in the interior of $G$ and $\chi(x)=0$ if $x$ in the exterior of $G$. 
 \end{proposition}
The way how the parenthesis are put is crucial again. Putting the parenthesis differently, leads in the left octonionic monogenic case to the different formula of the form
$$
\frac{3}{\pi^4} \int\limits_{\partial G} \Big( q_{\bf 0}(y-x) d\sigma(y)\Big) f(y) = \chi(x) f(x) + \int\limits_G \sum\limits_{i=0}^7 \Big[q_{\bf 0}(y-x),{\cal{D}}f_i(y),e_i  \Big] dy_0 \cdots dy_7, 
$$
again involving the associator, cf. \cite{XL2002}. The volume integral term appearing additionally always vanishes in associative algebras, such as in Clifford or quaternionic analysis.
\par\medskip\par
To round off this preliminary section we wish to emphasize that there also exist alternative powerful extensions of complex function theory to the octonionic setting. For instance there is the complementary theory of slice-regular octonionic functions which is essentially different from that of octonionic monogenic functions, although there are connections by Fueter's theorem or the Radon transformation. The classical approach (see \cite{StruppaGentili}) extends complex-analytic functions from the plane to the octonions by applying a radially symmetric model fixing the real line. More  recently, see for instance \cite{JRS} and \cite{Ghiloni2020}, one also started to study octonionic slice-regular extensions departing differently from monogenic functions that are defined in the quaternions. However, in this paper we restrict ourselves to entirely focus on the theory of octonionic monogenic functions, although we also expect that one can successfully establish similar results in the alternative framework of slice-regular functions in $\mathbb{O}$. Apart from octonionic monogenic function theory and slice-regular octonionic function theories there are even more possibilities for introducing further complementary function theories in octonions.  

\section{Main results}
Throughout this section let $\Omega \subset \mathbb{O}$ be a simply-connected orientable domain with a strongly Lipschitz boundary, say $\Sigma=\partial \Omega$, where the exterior normal field exists almost everywhere. Let us denote by $n(y)$ the exterior unit normal octonion at a point $y \in \partial \Omega$. 

Next, let $H^2(\partial \Omega, \mathbb{O})$ be the closure of the set of $L^2(\partial \Omega)$-octonion-valued functions that are left octonionic monogenic functions inside of $\Omega$ and that have a continuous extension to the boundary $\partial \Omega$. For a lot of interesting insight in the general study of octonionic Hilbert spaces we also refer the interested reader to \cite{L} where the focus is put on different aspects. 

\subsection{Attempts to define an octonionic monogenic Szeg\"o projection}
To introduce a meaningful generalization of a Hardy space in the octonionic setting, one first needs to define a properly adapted inner product. Being inspired by the preceding papers \cite{WL2018,WL2020} we first consider the following definition: 
\begin{definition} For any pair of octonion-valued functions $f,g \in L^2(\partial \Omega)$ one defines the following $\mathbb{R}$-linear octonion-valued inner product
\begin{eqnarray*}
(f,g)_{\partial \Omega} &:=& \frac{3}{\pi^4} \int\limits_{\partial \Omega}(\overline{n(x)g(x)})\; (n(x)f(x)) dS(x)\\
&=& \frac{3}{\pi^4} \int\limits_{\partial \Omega} (\overline{g(x)}\; \overline{n(x)})\; (n(x) f(x)) dS(x),
\end{eqnarray*}
where $dS(x)$ again represents the scalar Lebesgue surface measure on $\partial \Omega$.   
\end{definition}
\par\medskip\par
When it is clear to which domain we refer, we omit the subindex $\partial\Omega$ for simplicity. Note that the factor $\frac{3}{\pi^4}$ is part of this definition.  
By a direct calculation one observes that $(\cdot,\cdot)$ is $\mathbb{R}$-linear.  
For all octonion-valued functions $f,g,h \in L^2(\partial \Omega)$ and all $\alpha,\beta\in \mathbb{R}$ we have $(f+g,h) = (f,h) + (g,h)$ and $(\alpha f,g \beta) = \alpha(f,g)\beta$. Notice that in view of the lack of associativity $(\cdot,\cdot)$ is only $\mathbb{R}$-linear but not $\mathbb{O}$-linear. The lack of $\mathbb{O}$-linearity however implies some serious obstacles, since we cannot rely on a Cauchy-Schwarz inequality. Consequently, we cannot rely on a direct analogue of the Fischer-Riesz representation theorem using this  definition of inner product.  

We still may observe that $(\cdot,\cdot)$ is Hermitian in the sense of the octonionic conjugation, since 
\begin{eqnarray*}
\overline{(f,g)} &=&  \overline{\frac{3}{\pi^4} \int\limits_{\partial \Omega}(\overline{n(x)g(x)})\; (n(x)f(x)) dS(x)} \\
&=& \frac{3}{\pi^4} \int\limits_{\partial \Omega}(\overline{n(x)f(x)}) \; (n(x)g(x)) dS(x)\\
& = & (g,f). 
\end{eqnarray*}
 
One may also directly observe that 
$$
(f,f) = \frac{3}{\pi^4} \int\limits_{\partial \Omega} (\overline{f(x)}\; \overline{n(x)})\; (n(x)\; f(x)) dS(x) = \frac{3}{\pi^4} \int\limits_{\partial \Omega} |f(x)|^2 dS(x) = \|f\|^2_{L^2},
$$
since the product inside the integral is generated by only two elements $n(x)$ and $f(x)$ and hence it is  associative according to Artin's theorem.  
\par\medskip\par
Endowed with this inner product it is suggestive to call $H^2(\partial \Omega, \mathbb{O})$ the (left) octonionic monogenic Hardy space of $\Omega$ in some wider sense. Note that the term ``space'' has to be understood in the sense of a real vector space when using this octonion-valued inner product.  
\begin{remark}
Notice further that if we were in an associative setting (such as in complex or Clifford analysis), then one would have 
$$
(\overline{g(x)}\; \overline{n(x)})\; (n(x)\; f(x)) = \overline{g(x)} |n(x)|^2 f(x) = \overline{g(x)}f(x).
$$
So, one re-discovers the usual definition of the Hardy space inner product used in the classical framework. 
\par\medskip\par
In the octonionic setting the introduction of the normal field $n$ inside these brackets makes a difference. It allows us to recognize the Cauchy transform in the framework of this inner product. This in turn permits us to use the special properties of the Cauchy transform in this context. Furthermore, it gives us a hint what might be a meaningful octonionic monogenic definition of a generalization of the adjoint Cauchy transform (again in a wider sense) and how to define a compact Kerzman-Stein operator.  
\end{remark}
\begin{remark}
In the particular case where $\Omega = B_8(0,1)$ is the octonionic unit ball which has been addressed in {\rm \cite{WL2018}} one has exactly that $n(x) = x$. In this case the inner product $(\cdot,\cdot)$ simplifies to 
$$
(f,g)_{S_7} = \frac{3}{\pi^4} \int\limits_{\partial B_8(0,1)} (\overline{x \; g(x)}) \; (x \; f(x)) dS(x)
$$
and we re-obtain exactly the definition introduced in {\rm \cite{WL2018}}. 
\par\medskip\par
In the special case where $\Omega = H^{+}(\mathbb{O}) = \{ x \in \mathbb{O} \mid x_0 > 0\}$ is the octonionic half-space, one has even more simply that $n(x) = - 1$. Now the corresponding inner product reduces to 
$$
(f,g)_{H^+}  = \frac{3}{\pi^4} \int\limits_{\partial H^+} \overline{g(x)} f(x) dS(x) 
             = \frac{3}{\pi^4} \int\limits_{\mathbb{R}^7} \overline{g(x)} f(x) dx_1 dx_2 \cdots dx_7
$$
like in the classical associative cases of complex and Clifford analysis. The use of the usual inner product suggested for the treatment of the half-space in {\rm \cite{WL2020}} thus makes completely sense and fits well in that context. 
\end{remark}
A crucial question is to ask whether or not there always exists a reproducing kernel function. For the case of the unit ball, upper half-space and strip domains being bounded in the $x_0$-direction the existence has been shown by presenting an explicit kernel function that is octonionic monogenic in the first variable and octonionic anti-monogenic in the second variable which in fact turned out to reproduce all $f$ belonging to $H^2$ using exactly this definition of inner product. Here, one actually exploited the Cauchy integral formula. In order to apply Cauchy's integral formula, the presence of the normal field is indeed crucially important.  Actually, in the case of the unit ball, the octonionic Szeg\"o projection defined by 
$[{\cal{S}} f](y) := (f,S(\cdot,y))_{S_7}$ where $S(x,y) = \frac{1-\overline{x}y}{|1-\overline{x}y|^8}$ obviously is octonionic monogenic in $x$ and octonionic anti-monogenic in $y$, coincides with the Cauchy transformation when using exactly that definition, cf. \cite{WL2018}. A similar relation has been proved for the half-space $x_0 > 0$ where we have $[{\cal{S}} f](y) := (f,S(\cdot,y))_{H^{+}}$ with $S(x,y) = \frac{\overline{x}+y}{|\overline{x}+y|^8}$ for all $f \in H^2(\partial H^+)$. Finally, an analogous formula has been established for strip domains being bounded in the $x_0$-direction, see \cite{Kra2020-2}. 

However, for other domains the existence of an octonionic monogenic Szeg\"o kernel is far from being evident. This will be explained in what follows. 
\par\medskip\par
First we point out that it is still very easy to prove 
\begin{proposition}
Let $\Omega \subset \mathbb{O}$ be a general simply-connected orientable domain where the exterior unit normal exists almost everywhere.  
The set $H^2(\partial \Omega, \mathbb{O})$ equipped with the above mentioned octonion-valued inner product satisfies the Bergman condition.
\end{proposition}
\begin{proof}
Suppose that $\Omega \subset \mathbb{O}$ is an arbitrary bounded or unbounded orientable domain with a sufficiently smooth boundary and let $x \in \Omega$. Let $B_8(x,R)$ be the eight-dimensional open ball centered at $x$ with radius $R$ where one chooses $R > 0$ such that the solid ball $\overline{B_8(x,R)} \subset \Omega$. Then, relying on the version of the octonionic Cauchy integral given in Proposition~\ref{cauchy1} we get that 
\begin{eqnarray*}
|f(x)|^2 &=& \Bigg|  \frac{3}{\pi^4} \int\limits_{\partial \Omega} q_{\bf 0}(y-x) \; (n(y) f(y)) dS(y)\Bigg|^2\\
 & = & \Bigg|  \frac{3}{\pi^4} \int\limits_{\partial B_8(x,R)} q_{\bf 0}(y-x) \; (n(y) f(y)) dS(y)\Bigg|^2.
\end{eqnarray*}
Note that due to the absence of the $\mathbb{O}$-linearity we cannot simply apply the inequality of Cauchy-Schwarz like proposed in {\rm \cite{BDS}}. However, we may rely on Cauchy's integral formula and the norm property of the octonions which luckily ensure that     
\begin{eqnarray*}
|f(x)|^2 & \le & const(B_8(x,R)) \frac{3}{\pi^4} \int\limits_{\partial B_8(x,R)} |f(y)|^2 dS(y)\\
         & \le & const(B_8(x,R)) \|f\|^2_{L^2(\partial \Omega)},
\end{eqnarray*}  
where $const(B_8(x,R))$ is a constant which just depends on the domain. 
Hence, we indeed have a continuous point evaluation as a consequence of the validity of the octonionic Cauchy integral formula. 
\end{proof}
However, due to the absence of a direct analogue of a Fischer-Riesz representation theorem, we cannot directly use the previously established property to guarantee the existence of a reproducing kernel function.  
\par\medskip\par
One possible way to overcome this problem is to consider first the real-valued inner product defined by 
$$
 \langle f,g \rangle_0 := \Re\{(f,g)\}.
$$
Using this real-valued inner product instead one can apply all the well-known theorems and properties of a classical reproducing kernel Hilbert space. See also \cite{Cnops1996} where a similar idea has been applied to the Clifford analysis case.  The Clifford case however was easier to handle because the associativity permitted a complete equivalent treatment of the corresponding function spaces. Using the real-valued inner product or the Clifford-valued inner product is equivalent due to linearity and associativity. In the octonionic case we have to be more careful.  

In view of the property (\ref{scalarprod}) $\langle ab,c \rangle = \langle b, \overline{a}c\rangle$ for all $a,b,c \in \mathbb{O}$ we have an invariance of that real-valued inner product of the form $ \langle nf,ng \rangle_0$ for any $n$ with $|n|=1$. This is due to $n \overline{n} = 1$ and a consequence of the property that $[n,\overline{n},a] = 0$ for any $a,n \in \mathbb{O}$. So, the real-valued inner product can equivalently be rewritten in the form 
$$
\langle f,g \rangle_0 = \int\limits_{\partial \Omega} \Re\Big\{ \overline{g(x)} f(x)  \Big\} dS(x).
$$
The property (\ref{scalarprod}) also provides us with the necessary compatibility condition to define for any  octonionic $\mathbb{R}$-linear functional ${\cal{T}}$ a uniquely defined adjoint ${\cal{T}}^*$ such that $ \langle {\cal{T}}f,g \rangle_0 =  \langle f,{\cal{T}}^*g\rangle_0$. 
In particular, when equipping $H^2(\partial \Omega,\mathbb{O})$ now with this real-valued inner product $ \langle \cdot,\cdot \rangle_0$, one indeed always gets a uniquely defined reproducing kernel ${{S_0}}_x: y \mapsto {{S_0}}_x(y):={S_0}(x,y)$ in $H^2(\partial \Omega)$ called the octonionic Szeg\"o kernel with respect to $\langle \cdot , \cdot \rangle_0$. It reproduces the harmonic real part $f_0$ of an octonionic monogenic function $f$ 
$$
[{\cal{S}}_0 f](x) :=  \langle f,{{S_0}}_x \rangle_0 = f_0(x) \quad \quad \forall f \in H^2(\partial \Omega, \mathbb{O}).
$$  
Let us now investigate the reproduction behavior of the other components of an octonion-valued function $f(x) = f_0(x) + f_1(x) e_1 + \cdots + f_7(x) e_7$ belonging to $H^2(\partial \Omega)$. Here $f_i$ always denotes the real component of the part of $f$ belonging to $e_i$ also called the $i$-part of $f$. If $f \in H^2(\partial \Omega)$, then all the real components $f_i$ are harmonic functions from $\mathbb{O}$ to $\mathbb{R}$.  
Similarly, we define for all $i = 1,\ldots,7$ the real-valued inner products 
$$ \langle f,g \rangle_i := \{(f,g)\}_i = \int\limits_{\partial \Omega} \{(\overline{n(y)\;g(y)})\;(n(y)\;f(y))   \}_i dS(y),
$$ 
referring to the $i$-part of the octonionic expression $(f,g)$. For each $i=1,\ldots,7$ there exists a unique reproducing kernel ${{S_i}}_x: y \mapsto {{S_i}}_x(y):={{S_i}}(x,y)$ in $H^2(\partial\Omega)$ that exactly reproduces the harmonic $i$-part $f_i$ of an octonionic monogenic function $f \in H^2(\partial \Omega)$ - now with respect to the inner product $ \langle \cdot,\cdot \rangle_i$, i.e.:
$$
[{\cal{S}}_i f](x) :=  \langle f,{{S_i}}_x \rangle_i = f_i(x) \quad \quad \forall f \in H^2(\partial \Omega, \mathbb{O}). 
$$
Now we can introduce a total octonionic monogenic Szeg\"o projection by 
\begin{equation}\label{monszego}
[{\cal{S}} f] := \sum\limits_{i=0}^7 \langle f , S_i \rangle_i e_i.
\end{equation}
Per construction it satisfies $[{\cal{S}} f](x) = \sum\limits_{i=0}^7 f_i(x) e_i = f(x)$ for all $f\in H^2(\partial\Omega)$. The projection ${\cal{S}}_0 f$ is exactly its real part. An open question is to ask: When is it possible to represent the total Szeg\"o projection 
${\cal{S}}$ in a global form of the way  
\begin{equation}\label{globalform}
[{\cal{S}}f](x) := \frac{3}{\pi^4} \int\limits_{\partial \Omega}(\overline{n(y) S(x,y)})\; (n(y)f(y)) dS(y)?
\end{equation}
At this point we wish to recall that actually in the particular cases where $\Omega$ is the unit ball, the half-space or a strip domain bounded in the real direction we can in fact write ${\cal{S}}f$ can in such a global form involving an explicit kernel function $S(x,y)$ reproducing all $f \in H^2(\partial \Omega)$ as in all detail mentioned above. 

However, the projection ${\cal{S}}$ does not always have the usual properties that we are used to observe from the classical Szeg\"o projection. Thus, in the general case one has to reproduce componentwisely by using the real-valued inner products instead.    
\par\medskip\par
Note that in view of the Hermitian property one has the symmetric relation $$[{{S_i}}(y,x)]_i = [{{S_i}}_y(x)]i =  \langle {{S_i}}_y,{{S_i}}_x\rangle_i = \overline{ \langle {{S_i}}_x,{{S_i}}_y \rangle_i} = \overline{[{{S_i}}_x(y)]_i} = \overline{[{{S_i}}(x,y)]_i}$$ for all $i$-parts from $i=0,\ldots,7$.  
\par\medskip\par
A simple but very important property is the following 
\begin{corollary}
The total Szeg\"o projection defined in (\ref{monszego}) is an $\mathbb{R}$-linear operator, satisfying 
$$
{\cal{S}}[\alpha f + \beta g]=\alpha {\cal{S}} [f] + \beta {\cal{S}} [g]
$$
for all octonion-valued $f,g \in L^2(\partial \Omega)$ and all $\alpha,\beta \in \mathbb{R}$. 
\end{corollary}
\begin{proof}
By using the definition we directly may read off that 
\begin{eqnarray*}
{\cal{S}}[\alpha f + \beta g] &=& \sum\limits_{i=0}^7 \langle \alpha f, S_i\rangle _i e_i + \sum\limits_{i=0}^7
 \langle \beta g, S_i\rangle_i e_i \\
&=& \alpha \sum\limits_{i=0}^7 \langle f, S_i\rangle_i e_i + \beta \sum\limits_{i=0}^7 \langle g, S_i \rangle_i e_i \\
&=& {\cal{S}}[f] + \beta {\cal{S}}[g].
\end{eqnarray*}
\end{proof}
\par\medskip\par
Let us now analyze whether one can associate with the Szeg\"o projection some notions of orthogonality. 
 
We first may observe that even in the context of the octonion-valued inner product $(\cdot,\cdot)$ there exists a uniquely defined projection ${\cal{P}}: L^2(\partial \Omega) \to H^2(\partial \Omega, \mathbb{O})$ such that $(f-{\cal{P}}f,g) = 0$ for all $f \in L^2(\partial \Omega)$ and $g \in H^2(\partial \Omega)$. Since $(f+h,g)=(f,g)+(h,g)$ for all $f,h \in L^2(\partial\Omega)$ and $g \in H^2(\partial \Omega)$  one has that $(f-{\cal{P}}f,g) = 0$ if and only if $(f,g)=({\cal{P}}f,g)$. However, the absence of a direct analogue of the Fischer-Riesz representation theorem does not immediately guarantee us the possibility to express the projection ${\cal{P}}$ in terms of a global Szeg\"o kernel. However, the consideration of the real-valued inner products $\langle \cdot, \cdot \rangle_i$ allow us link the Szeg\"o projection with the notion of orthogonality and self-adjointness in the context of $\langle \cdot, \cdot \rangle_i$. 

A simple but rather important consequence of the preceding corollary is the following  
\begin{corollary}
The Szeg\"o operator ${\cal{S}}$ possesses an $\mathbb{R}$-linear adjoint operator operator ${\cal{S}}^*$ such that $\langle {\cal{S}} f,g\rangle_0 = \langle f,{\cal{S}}^*g\rangle_0$ for all $f,g \in L^2(\partial \Omega)$. (Similarly, for the other $i=1,\ldots,7$ there are $\mathbb{R}$-linear adjoint operators ${\cal{S}}^{[i*]})$ satisfying $\langle {\cal{S}} f,g \rangle_i = \langle f , {\cal{S}}^{[*i]}g\rangle_i$ for all $f,g \in L^2(\partial \Omega)$. 
\end{corollary} 
In the setting of the real-valued inner product $\langle f , g \rangle_0$ we can talk about self-adjointness of an $\mathbb{R}$-linear octonionic integral operator ${\cal{T}}$ when $\langle {\cal{T}} f , g \rangle_0 = \langle f , {\cal{T}} g \rangle_0$.  
Let us now assume that ${\cal{T}}f$ is representable with a global kernel function $k(x,y)$. Self-adjointness w.r.t. $\langle \cdot, \cdot \rangle_0$ happens when the associated kernel function of ${\cal{T}}^*$ is the conjugate of the kernel function of ${\cal{T}}$ in view of the compatibility relation $\langle u v,w\rangle = \langle v , \overline{u}w\rangle$. Here, the usual results from the theory of Hilbert spaces can be applied, and we may talk about an orthogonal projector.  

Let us next assume that the Szeg\"o projection ${\cal{S}}$ is representable with a kernel function $S(x,y)$ --- we know that this is at least true for the three cases of the unit ball, the half-space and the strip domain. Then the associated adjoint ${\cal{S}}^*$ (w.r.t. $\langle \cdot, \cdot \rangle_0$)  is represented in terms of the conjugate of the kernel, i.e. $\overline{S(y,x)}$. We have ${\cal{S}}^*={\cal{S}}$ if and only if $\overline{S(y,x)} = S(x,y)$. This is true at least in the cases of $\Omega = B_8(0,1)$, $\Omega = H^{+}$ or $\Omega = \{z \in \mathbb{O} \mid 0 < \Re(z) < d\}$. There, we know for sure that ${\cal{S}}$ is self-adjoint w.r.t. $\langle \cdot, \cdot \rangle_0$.

However, one has to be extremely careful concerning the introduction of a notion of self-adjointness when looking at the full octonionic inner product $(\cdot,\cdot)$.  

Suppose now that we also would have the relation $({\cal{S}}f,g)=(f,{\cal{S}}g)$ for all octonionic valued functions $f,g$ belonging to $L^2(\partial \Omega)$. Assume further that we could represent ${\cal{S}}$ in a global form of the way (\ref{globalform}) which a global octonionic monogenic Szeg\"o kernel $S(x,y)$. 

That would mean that despite of the non-associativity we always would have the property that 
\begin{equation}
\label{adjointrelationszego}
\int\limits_{\partial \Omega} \overline{(n(x)g(x))} \; (n(x){\cal{S}}f(x)) dS(x) = 
\int\limits_{\partial \Omega} \overline{(n(x) {\cal{S}}g(x))}\; (n(x)f(x)) dS(x).
\end{equation}  
In turn, this would imply that for any $f,g \in L^2(\partial \Omega)$ one would get the identity  
\begin{eqnarray*}
& & \int\limits_{\partial \Omega} (\overline{n(x)g(x)}) \; \Bigg(n(x) \; \Bigg[     
\int\limits_{\partial \Omega}(\overline{n(y)S(x,y)})(n(y)f(y)) dS(y)
\Bigg]\Bigg) dS(x)\\
 & \stackrel{(\ref{adjointrelationszego})}{=} & \int\limits_{\partial \Omega}(\overline{{\cal{S}}g(x)} \;  \overline{n(x)})\; (n(x)f(x)) dS(x)  \\
&=& \int\limits_{\partial \Omega} \Bigg( \Bigg[\overline{\int\limits_{\partial\Omega}(\overline{n(y)S(x,y)})        \;(n(y)g(y)) dS(y)}\Bigg]\overline{n(x)} \Bigg)\; (n(x)f(x)) dS(x).
\end{eqnarray*}
So, in particular for the half-space case where $n(x)=-1$ and where we know that ${\cal{S}}$ is self-adjoint w.r.t. $\langle \cdot, \cdot \rangle_0$ one would get that 
$$
\int\limits_{\mathbb{R}^7} \overline{g(x)} \; [{\cal{S}}f](x) dx_1 \cdots dx_7 
= \int\limits_{\mathbb{R}^7} \overline{[{\cal{S}}g](x)} \; f(x) dx_1 \cdots dx_7.
$$
This means in detail that one would get for any $f$ and $g$ the relation
\begin{eqnarray*}
 & & \int\limits_{\mathbb{R}^7} \overline{g(x)} \Bigg(\int\limits_{\mathbb{R}^7} \overline{S(x,y)} f(y) dy_1 \cdots dy_7 \Bigg) dx_1 \cdots dx_7\\
&=& \int\limits_{\mathbb{R}^7} \Bigg(\overline{\int\limits_{\mathbb{R}^7} \overline{S(x,y)} g(y) dy_1 \cdots dy_7    }\Bigg) f(x) dx_1 \cdots dx_7\\
&=& \int\limits_{\mathbb{R}^7} \Bigg(\int\limits_{\mathbb{R}^7} \overline{g(y)}S(x,y) dy_1 \cdots dy_7   \Bigg) f(x) dx_1 \cdots dx_7\\
&=& \int\limits_{\mathbb{R}^7} \Bigg(\int\limits_{\mathbb{R}^7} \overline{g(y)} \; \overline{S(y,x)} dy_1 \cdots dy_7    \Bigg) f(x) dx_1 \cdots dx_7.
\end{eqnarray*}
However, note that we do not have in general a termwise associativity. Therefore, such a relation could not be expected in such generality.

\subsection{The octonionic Cauchy projection revisited in the context of these inner products}
Closely related to the Szeg\"o projection there is also the Cauchy projection induced by the octonionic Cauchy integral formula. Suppose that $f \in L^2(\partial \Omega)$. Then the octonionic Cauchy projection
\begin{eqnarray*}
[{\cal{C}} f](x) &:=& \frac{3}{\pi^4} \int\limits_{\partial \Omega} q_{\bf 0}(y-x) \; (d\sigma(y) f(y)) \\
&=& \frac{3}{\pi^4} \int\limits_{\partial \Omega} q_{\bf 0}(y-x) \; (n(y) f(y)) dS(y)
\end{eqnarray*}
sends an $L^2(\partial \Omega)$-function to a function belonging to $H^2(\partial \Omega,\mathbb{O})$ for any $x \in \Omega$, cf. \cite{XLT2008}, Theorem~2.2.  
 
The octonionic Cauchy projection can easily be re-written in terms of the octonionic inner product defined at the beginning of the previous subsection in the following form 
$$
[{\cal{C}}f](x) = (f,g_x) = \frac{3}{\pi^4} \int\limits_{\partial \Omega} (\overline{n(y) g_x(y)})\; (n(y)f(y))dS(y)
$$
where we identify $\overline{n(y) g_x(y)} = q_{\bf 0}(y-x) = \frac{\overline{y-x}}{|y-x|^8}$. 

Thus, $n(y) g_x(y) = \frac{y-x}{|y-x|^8}$ from which we may read off that 
$$
g_x(y) = \overline{n(y)} \frac{y-x}{|y-x|^8}. 
$$ 
In this notation and using the special octonion-valued inner product the well-known octonionic Cauchy integral formula can be re-expressed in the form 
$$
(f,g_x) = f
$$
if $f$ is left octonionic monogenic. 
\par\medskip\par
Let us fix some notation. Suppose that $\mu$ is a real number satisfying $0 < \mu < 1$. Let us denote the set of octonion-valued functions that are $\mu$-h\"oldercontinuous over $\partial \Omega$ by $C^{\mu}(\partial \Omega)$, similarly as in complex analysis.  
Next, following the paper \cite{XLT2008}~Theorem 2.2, also in line with the results from classical complex and Clifford analysis, for any $f \in C^{\mu}(\partial \Omega)$ the octonionic Cauchy transform can be extended to the boundary by defining
$$
[{\cal{C}}f](x) = \frac{1}{2} f(x) + \frac{3}{\pi^4} p.v. \int\limits_{\partial \Omega} q_{\bf 0}(y-x) \; (d\sigma(y) f(y))
$$
where 
$$
p.v. \int\limits_{\partial \Omega} q_{\bf 0}(y-x) \; (d\sigma(y) f(y)) := \lim\limits_{\varepsilon \to 0^+} \int\limits_{\partial \Omega,\;|x-y|\ge \varepsilon} q_{\bf 0}(y-x) \; (d\sigma(y) f(y))
$$ 
is the Cauchy principal value. 

The second term represents the octonionic monogenic generalized Hilbert transform which will be denoted by 
$$
[{\cal{H}}f](x) := 2 p.v. \frac{3}{\pi^4} \int\limits_{\partial \Omega} q_{\bf 0}(y-x) \; (d\sigma(y) f(y)).
$$
Equivalently, writing 
$$
[{\cal{P}}_{\pm} f](x) = \pm \lim\limits_{\delta \to 0^+} \frac{3}{\pi^4} \int\limits_{\partial \Omega} q_{\bf 0}(y-x\pm \delta) \; (d\sigma(y) f(y)),
$$
one deals with the Plemelj projectors 
$$
{\cal{P}}_+ = \frac{1}{2} ({\cal{H}} + {\cal{I}}),\quad\quad {\cal{P}}_- = \frac{1}{2} (-{\cal{H}} + {\cal{I}})
$$
where ${\cal{I}}$ is the identity operator acting in the way ${\cal{I}}f = f$. One obtains the Plemelj projection formulas ${\cal{P}}_+  + {\cal{P}}_- = {\cal{I}}$ and ${\cal{P}}_+ - {\cal{P}}_- = {\cal{H}}$. 

The extended octonionic monogenic Cauchy transform ${\cal{C}}: C^{\mu}(\partial \Omega) \to H^2(\partial \Omega, \mathbb{O})$ satisfies like in the complex case ${\cal{C}}^2 = {\cal{C}}$. Let $f \in C^{\mu}(\partial \Omega)$. Then $g :={\cal{C}} [f] \in H^2(\partial \Omega, \mathbb{O})$. Now, also the octonionic calculation rules allow us to conclude that $[{\cal{C}}^2]f = {\cal{C}}[{\cal{C}}[f]] = {\cal{C}}[g]=g = {\cal{C}}[f]$.   

Furthermore, one has $\|{\cal{H}}f\|_{L_2} \le c \|f\|_{L_2}$ with a real positive constant $c$, since $C^{\mu}(\partial \Omega)$ is dense in $L^2(\partial \Omega)$. Consequently, $\|{\cal{C}} f\|_{L_2} \le (\frac{1}{2}+c)\|f\|_{L_2}$. Therefore ${\cal{H}}$ and ${\cal{C}}$ are both $L^2$-bounded operators. 

\begin{remark}
In contrast to Clifford analysis, the octonionic Cauchy transform is only $\mathbb{R}$-linear and not $\mathbb{O}$-linear in general. Due to the lack of associativity, in general $[{\cal{C}}(f \alpha)] \neq [{\cal{C}}f] \alpha$ if $\alpha \not\in \mathbb{R}$, because
$$
q_{\bf 0}(y-x) \; \Bigg( d\sigma(y) \; (f(y) \; \alpha)\Bigg) \neq \Bigg(q_{\bf 0}(y-x) \; ( d\sigma(y) \; f(y))\Bigg)\; \alpha. 
$$
We only have ${\cal{C}}[f\alpha + g \beta] = [{\cal{C}}f]\alpha + [{\cal{C}}g] \beta$ for real $\alpha,\beta$. However, this property is sufficient to guarantee the existence of a unique adjoint ${\cal{C}}^*$ satisfying $\langle {\cal{C}}f , g \rangle_0 = \langle f, {\cal{C}}^{*}g\rangle_0$ for all $f,g \in C^{\mu}(\partial \Omega)$. This is a consequence of the Fischer-Riesz representation theorem which can be applied in the context of a real-valued inner product, for instance for $\langle \cdot, \cdot,\rangle_0$. 
\end{remark}

In the notation of our previously defined octonion-valued inner product on $L^2(\partial \Omega)$, the extended octonionic Cauchy transform can be re-expressed in the form
$$
[{\cal{C}} f](x) = \frac{1}{2}f(x) + p.v. \frac{3}{\pi^4} \int\limits_{\partial \Omega} (\overline{n(y) g_x(y)})\; (n(y)f(y))dS(y)
$$
where $g_x(y) = \overline{n(y)} \frac{y-x}{|y-x|^8}$. 

This representation gives us a hint how a meaningfully defined generalized dual octonionic monogenic Cauchy transform ${\cal{C}}^*$ on a dual-like function space could look like.  

Let us recall once more that in the case where $\Omega = B_8(0,1)$, the octonionic Cauchy transform ${\cal{C}}$ even coincides exactly with the global Szeg\"o projection ${\cal{S}}$ induced by the octonion-valued inner product $(f,S_x)$ where $S_x(y) = \frac{1-\overline{x}y}{|1-\overline{x}y|^8}$ is the uniquely defined octonionic monogenic Szeg\"o kernel of the unit ball. Remind also that whenever we wish to talk about self-adjointness and orthogonality we again have to switch to the real-valued inner products.

In the case of the unit ball (and only in this case) the octonionic Cauchy transform is self-adjoint in the sense of the real-valued inner product $\langle \cdot, \cdot \rangle_0$ in view of  ${\cal{C}}^* = {\cal{S}}^* = {\cal{S}} = {\cal{C}}$. In all the other cases, however, the octonionic Cauchy-transform is not self-adjoint, because it is not an orthogonal projector.    

Since ${\cal{C}}f = (f,g_x)$ it is suggestive to introduce an octonionic generalization of the dual Cauchy transform on the dual space in terms of the conjugated integral kernel $\overline{g_y(x)}$. Since $g_x(y) = \overline{n(y)} \frac{y-x}{|y-x|^8}$ we have $g_y(x) = \overline{n(x)} \frac{x-y}{|x-y|^8}$ and hence
$$
\overline{g_y(x)} =\frac{\overline{x-y}}{|x-y|^8} n(x).
$$
Thus, it is natural to define  
\begin{definition} (generalized dual octonionic Cauchy transform)\\
The generalized dual octonionic monogenic Cauchy transform is defined by  
$$
{\cal{C}}^*: C^{\mu}(\partial \Omega) \to L^2(\partial \Omega): \quad [{\cal{C}}^* f](x) = \frac{1}{2} f(x) + p.v. \frac{3}{\pi^4} \int\limits_{\partial \Omega} (\overline{n(y) \overline{g_y(x)}}) \; (n(y) f(y)) dS(y) = (f,\overline{g_y}). 
$$
\end{definition}
\begin{remark}
Due to the lack of a termwise associativity it is not clear on the basis of standard arguments whether one could expect a general relation of the form $({\cal{C}}f,g)=(f,{\cal{C}}^*g)$ for all h\"oldercontinuous octonion-valued functions $f,g$ defined over $\partial \Omega)$. The direct standard proof presented in {\rm \cite{Bell,Ca,PVL}} for the complex setting (resp. for the Clifford analysis setting) cannot be carried over since we cannot interchange the parenthesis due to the absence of associativity. However, it is rather easy to see that this relation at least holds for some particular cases where we have $f=g$. In the case where $\Omega$ is bounded one can simply take $f=g=1$. 
However, if we work with the real part of the inner product, then the existence and the uniqueness of the octonionic adjoint operator ${\cal{C}}^*$ is guaranteed by the Fischer-Riesz representation theorem and we can conclude that this integral kernel actually induces the adjoint Cauchy transform in all cases. In fact from $({\cal{C}} 1,1) = (1,{\cal{C}}^*1) = \overline{(1,{\cal{C}} 1)}$ it compulsively follows that the integral kernel of ${\cal{C}}^*$ must be the conjugate of the kernel of ${\cal{C}}$ in view of the compatibility formula $\langle uv,w \rangle = \langle v, \overline{u}w\rangle$.   
\end{remark} 
\begin{remark}
Since ${\cal{C}}$ is $\mathbb{R}$-linear, continuous and bounded, the same is true for the previously introduced dual transform. Since $C^{\mu}(\partial\Omega)$ is dense in $L^2(\partial \Omega)$ we have that 
$\|{\cal{C}} f\|_{L^2} = \|{\cal{C}}^{*} f\|_{L^2}$.
\end{remark}
\subsection{An octonionic Kerzman-Stein operator}
Now we are in position to define meaningfully
\begin{definition}(octonionic Kerzman-Stein kernel)\\
Let $\Omega \subset \mathbb{O}$ a domain with the above mentioned conditions. 
For all $x,y \in \partial \Omega \times \partial \Omega$ with ($x \neq y$) the octonionic Kerzman-Stein kernel is given by 
$$
A(x,y) := g_x(y) - \overline{g_y(x)} = \overline{n(y)} \frac{y-x}{|y-x|^8} - \frac{\overline{x-y}}{|y-x|^8} n(x).
$$
\end{definition}
In the special case where one has $g_x(y) = \overline{g_y(x)}$ one gets exactly that $A(x,y) \equiv 0$. We will see that this will exactly happen if and only if $\Omega = B_8(0,1)$,  
providing us with a complete analogy to the complex case, \cite{Bell}. Only in this situation the octonionic monogenic Cauchy transform turns out to satisfy ${\cal{C}}^* = {\cal{C}}$.   

The Kerzman-Stein kernel measures in a certain sense how much the domain $\Omega$ differs from the octonionic unit ball.  

We define the associated octonionic Kerzman-Stein operator ${\cal{A}}: L^2(\partial \Omega) \to L^2(\partial \Omega)$ by 
\begin{eqnarray*}
[{\cal{A}} f](x) &:=& (f,A_x)_{\partial \Omega}\\
								 & =& \frac{3}{\pi^4} \int\limits_{\partial \Omega}(\overline{n(y) A(x,y)}) \; (n(y)f(y))dS(y)\\
								&=& \frac{3}{\pi^4} \int\limits_{\partial \Omega} (\overline{A(x,y)}\;\overline{n(y)}) \; (n(y)f(y)) dS(y)
\end{eqnarray*}
Note that this is not a singular integral operator anymore and that it can be defined for all $f \in L^2(\partial\Omega)$. However, the two additive components are singular. That means that whenever we want to split these terms, then we again have to apply the Cauchy principal value and to work for instance in the subspace of $\mu$-h\"oldercontinuous functions: 
\begin{eqnarray*}
[{\cal{A}} f](x) &=&  \frac{3}{\pi^4} p.v. \int\limits_{\partial \Omega} (\overline{n(y) g_x(y)}) \; (n(y) f(y)) dS(y)\\
								&-& \frac{3}{\pi^4} p.v. \int\limits_{\partial \Omega}(\overline{n(y) \overline{g_y(x)}}) \; (n(y) f(y)) dS(y)\\
								& = & \frac{3}{\pi^4} p.v. \int\limits_{\partial \Omega} (\overline{n(y) g_x(y)}) \; (n(y) f(y)) dS(y) + \frac{1}{2} f(y)\\
&-& \frac{3}{\pi^4} p.v. \int\limits_{\partial \Omega}(\overline{n(y) \overline{g_y(x)}}) \; (n(y) f(y)) dS(y)	- \frac{1}{2} f(x). 							
\end{eqnarray*}
Since the octonions still offer the particular calculation rule $n(\overline{n} q) = (n \overline{n})q = q$, as we did explain in the preliminary section, the previous equation can be rewritten as 
\begin{eqnarray*}
[{\cal{A}} f](x) &=& \frac{1}{2}f(x) + p.v. \frac{3}{\pi^4} \int\limits_{\partial \Omega} q_{\bf 0}(y-x) \; (n(y)f(y)) dS(y)\\
&-& \frac{1}{2}f(x) - p.v. \frac{3}{\pi^4} \int\limits_{\partial \Omega} 
\Bigg[\Bigg(\overline{n(x)} \frac{x-y}{|x-y|^8}   \Bigg)\; \overline{n(y)}\Bigg] \; (n(y)f(y)) dS(y) \\
&=& [{\cal{C}} f](x) - [{\cal{C}}^{*} f](x). 
\end{eqnarray*}
\begin{remark}
Our octonionic Kerzman-Stein kernel satisfies 
$$
\overline{A(y,x)} = \frac{\overline{x-y}}{|x-y|^8} n(x) - \overline{n(y)} \frac{y-x}{|x-y|^8} = - A(x,y).
$$
for all $(x,y) \in \partial \Omega \times \partial \Omega$ with $x\neq y$. 
Also the octonionic Kerzman-Stein operator ${\cal{A}}$ is skew symmetric, i.e. ${\cal{A}}^* = - {\cal{A}}$. It is bounded since $\|{\cal{A}}\|_{L_2} \le \|{\cal{C}}\|_{L_2} + \|{\cal{C}}^{*}\|_{L_2} = 2 \|{\cal{C}}\|_{L_2} \le L \|f\|_{L_2}$ with a real $L > 0$. Since ${\cal{C}}$ and also ${\cal{C}}^*$ are $\mathbb{R}$-linear and continuous, ${\cal{A}}$ is a compact operator since it is $L^2(\partial \Omega)$-bounded. 
\end{remark}
\begin{remark}
Also in the octonionic setting one can write the octonionic Kerzman-Stein operator in terms of the Hilbert transform as 
$$
{\cal{A}} = \frac{1}{2} {\cal{H}} - \frac{1}{2} {\cal{H}}^* 
$$ 
where we identity 
$$
[{\cal{H}}^*f](x) := 2\frac{3}{\pi^4} p.v. \int\limits_{\partial \Omega}(\overline{n(y) \overline{g_y(x)}}) \; (n(y) f(y)) dS(y), \quad f \in C^{\mu}(\partial \Omega)
$$
with the formal analogue of the adjoint Hilbert transform, establishing an analogy to the Clifford analysis setting. Compare with {\rm \cite{BD,Delanghe}}. In fact, using the real-valued inner product one has $\langle {\cal{H}}f , g \rangle_0 = \langle f, {\cal{H}}^*g\rangle_0$ when working in the spaces of $\mu$-h\"oldercontinuous functions. 
\end{remark} 
Also in the octonionic case we have 
\begin{corollary} The octonionic Kerzman-Stein kernel vanishes identically if and only if the domain $\Omega$ is the octonionic unit ball.
\end{corollary}
\begin{proof}
If $\Omega = B_8(0,1)$, then $n(x)=x$, $n(y)=y$ and $|x|^2=|y|^2=1$. Then $A(x,y)$ simplifies to 
$$
A(x,y) = \frac{\overline{y}(y-x)-(\overline{x}-\overline{y})x}{|y-x|^8} = \frac{|y|^2-\overline{y}x - |x|^2+\overline{y}x}{|y-x|^8}=0. 
$$
Conversely, if $A(x,y) \equiv 0$, then 
$$
\overline{n(y)}(y-x) = (\overline{x-y}) n(x).
$$
This relation however can only be true if $\Omega$ is the octonionic unit ball, cf. Lemma~12 of \cite{PVL}. The argument of \cite{PVL} can be used because the above mentioned expressions only consist of products of two octonions. Consequently, in view of Artin's theorem the lack of associativity does not affect the argumentation.  
\end{proof}
So, also in the octonionic case we have ${\cal{C}}f = {\cal{C}}^*f$ if and only if $\Omega$ is the unit ball.\par\medskip\par 
Another very special case represents again the setting where $\Omega$ is  the octonionic half-space $x_0 > 0$. Here, we have 
\begin{theorem}
If $\Omega = H^+(\mathbb{O})$, then the octonionic Kerzman-Stein operator represents the classical Hilbert-Riesz transform in the $x_0$-direction, i.e.
$$
[{\cal{A}} f](x) = 2 p.v. \int\limits_{\mathbb{R}^7} \frac{y_0-x_0}{|y-x|^8} f(y) dy_1 \cdots dy_7.
$$
\end{theorem}
\begin{proof}
If $\Omega = H^+(\mathbb{O})$ then $\partial \Omega = \mathbb{R}^7$ and $n(x) \equiv -1$. So, the octonionic Kerzman-Stein transformation simplifies to 
\begin{eqnarray*}
[{\cal{A}}f](x) &=& \int\limits_{\mathbb{R}^7} \overline{A(x,y)} f(y) dy_1 \cdots dy_7\\
&=& p.v. \int\limits_{\mathbb{R}^7} \frac{\overline{y-x}}{|y-x|^8} f(y) dy_1 \cdots dy_7 - p.v.  \int\limits_{\mathbb{R}^7} \frac{x-y}{|y-x|^8} f(y) dy_1 \cdots dy_7\\
&=& p.v. \int\limits_{\mathbb{R}^7} \frac{\overline{y-x}}{|y-x|^8} f(y) dy_1 \cdots dy_7 + p.v.  \int\limits_{\mathbb{R}^7} \frac{y-x}{|y-x|^8} f(y) dy_1 \cdots dy_7\\
&=& 2 p.v. \int\limits_{\mathbb{R}^7} \frac{\Re(y-x)}{|y-x|^8} f(y) dy_1 \cdots dy_7.
\end{eqnarray*}
\end{proof}
This provides us again with another nice analogy to the Clifford analysis setting, compare again with \cite{BD,Delanghe}. 

\section{Open problems and perspectives}

In the preceding section we have seen that an explicit description of the Szeg\"o projection for arbitrary Lipschitz domains with strongly Lipschitz boundaries is a very hard task on the one-hand. However on the other-hand, the Cauchy projection is very simple to describe namely in the global form $(f,g_x) = f$ for all $f \in L^2(\partial \Omega)$.   
Furthermore, note that the Cauchy kernel $q_{\bf 0}(y-x)$ has always the same representation; it is fully independent of the geometry of the domain and it is a global entity. The geometry of the domain is fully encoded in terms of the normal field $n(y)$. 

So, in the octonionic setting it is a lot easier to work with the Cauchy projection than with the Szeg\"o projection. A big goal would consist in establishing an approximation for the Szeg\"o projection in terms of the much simpler Cauchy projection. 
\par\medskip\par
We conclude this paper by formulating the following open problem: 
\par\medskip\par
Let us suppose that we have a domain with $\|{\cal{A}}\|_{L_2} < 1$. 
Is it possible under this condition to establish (alike in complex and Clifford analysis) an asymptotic relation of the form
$$
{\cal{S}} \approx {\cal{C}}\sum\limits_{j=0}^{N} (-{\cal{A}})^j \;\;?
$$
Of course, in this context we need to consider the real-valued inner product $\langle \cdot,\cdot\rangle_0$.  
\par\medskip\par
Note that it is immediate to see that  $({\cal{S}} {\cal{C}})[f](x) := 
{\cal{S}}[{\cal{C}}f](x) = {\cal{C}} [f](x)$ for all $f \in L^2(\partial \Omega)$ and all $x \in \Omega$ because for any $f \in L^2(\partial \Omega)$ the Cauchy projection ${\cal{C}} [f](x)$ turns always out to be an element of $H^2(\partial\Omega)$. Then the total Szeg\"o projection reproduces every  element of $H^2(\partial \Omega)$, so the property ${\cal{S}}{\cal{C}} = {\cal{C}}$ holds. 
\par\medskip\par
However, does there also always hold that ${\cal{S}}f \in H^2(\partial \Omega)$ for any $f \in L^2(\partial \Omega)$?  
Maybe one needs to put some restrictions or additional requirements to guarantee this property. 
\par\medskip\par
Note that the proof of the classical result also uses the relation ${\cal{C}}^* {\cal{S}} = {\cal{C}}^*$. 
When talking about the adjoint we of course need again to consider the framework of the real-valued inner product $\langle \cdot,\cdot\rangle_0$. However, we were not able to clarify in a satisfactory way yet, under which circumstances we do exactly have that ${\cal{S}}^*={\cal{S}}$ --- at least in the context of $\langle \cdot,\cdot\rangle_0$. 

As soon as we manage to give satisfactory answers to these questions then we will have all tools in hand to transfer the classical proof that is  applied in \cite{Bell,PVL} for the complex and Clifford analysis setting to express asymptotically the complicated total Szeg\"o projection in terms of the Cauchy projector and our compact Kerzman-Stein operator. This would allow us to deviate the complicated evaluation of the total Szeg\"o projection by using the simpler and globally valid Cauchy projection instead. This also would be a very nice application of the octonionic Kerzman-Stein operator in the resolution of octonionic boundary value problems in these function spaces. Results in this direction would open the door to tackle many interesting computational problems in $\mathbb{R}^8$. 

\par\medskip\par

{\bf Acknowledgement}. The authors are very thankful for the referees' valuable comments that lead to a significant improvement of this paper. 
 
\par\medskip\par
 
{\bf Data Availability Statement}. Data sharing not applicable to this article as no datasets were generated or analysed during the current study.

\end{document}